\documentclass[12pt]{article}%
\usepackage{amsmath}
\usepackage{amsfonts}
\usepackage{amssymb}
\usepackage{graphicx}%
\makeatletter
\newfont{\footsc}{cmcsc10 at 8truept}
\newfont{\footbf}{cmbx10 at 8truept}
\newfont{\footrm}{cmr10 at 10truept}
\newtheorem{theorem}{Theorem}

\newtheorem{corollary}[theorem]{Corollary}

\newtheorem{lemma}[theorem]{Lemma}

\newenvironment{proof}[1][Proof]{\noindent{\textbf {#1}  }}  {\hfill$\Box$\bigskip}

\def\blfootnote{\xdef\@thefnmark{}\@footnotetext}
\begin{document}

\title{Maxima of the $Q$-index: graphs with bounded clique number}
\author{Nair Maria Maia de Abreu\thanks{Federal University of Rio de Janeiro, Brazil;
\textit{email: nair@pep.ufrj.br }} \thanks{Research supported by CNPq Grant PQ
300563/94-9(NV).}\ and Vladimir Nikiforov\thanks{Department of Mathematical
Sciences, University of Memphis, Memphis TN 38152, USA; email:
\textit{vnikifrv@memphis.edu}} \thanks{Research supported by NSF Grant
DMS-0906634.}\bigskip}
\date{}
\maketitle

\begin{abstract}
This paper gives a tight upper bound on the spectral radius of the signless
Laplacian of graphs of given order and clique number. More precisely, let $G$
be a graph of order $n,$ let $A$ be its adjacency matrix, and let $D$ be the
diagonal matrix of the row-sums of $A.$ If G has clique number $\omega,$ then
the largest eigenvalue $q\left(  G\right)  $ of the matrix $Q=A+D$ satisfies
\[
q\left(  G\right)  \leq2\left(  1-1/\omega\right)  n.
\]
If $G$ is a complete regular $\omega$-partite graph, then equality holds in
the above inequality.\medskip

This result confirms a conjecture of Hansen and Lucas.\medskip\bigskip

\textbf{Keywords: }\emph{signless Laplacian; spectral radius; clique; cliqie
number; eigenvalue bounds.}

\textbf{AMS classification: }05C50

\end{abstract}

\section{Introduction}

Given a graph $G,$ write $A$ for its adjacency matrix and let $D$ be the
diagonal matrix of the row-sums of $A,$ i.e., the degrees of $G.$ The matrix
$Q\left(  G\right)  =A+D,$ called the \emph{signless Laplacian} or the
$Q$-matrix of $G,$ has been intensively studied recently, see, e.g., the
survey of Cvetkovi\'{c} \cite{C10} and its references.

We shall write $\mu\left(  G\right)  $ and $q\left(  G\right)  $ for the
spectral radii of the adjacency matrix and the signless Laplacian of a graph
$G;$ note in particular that $q\left(  G\right)  $ is known as the
$Q$\emph{-index} of $G$ and this is the meaning used in the title of the paper.

A pioneering result of Cvetkovi\'{c} \cite{Cve72} states that if $G$ is a
graph of order $n,$ with chromatic number $\chi\left(  G\right)  =\chi,$ then
\begin{equation}
\mu\left(  G\right)  \leq\left(  1-1/\chi\right)  n. \label{Ci}%
\end{equation}
Subsequently, Wilf \cite{Wil86} strengthened this inequality to%
\[
\mu\left(  G\right)  \leq\left(  1-1/\omega\right)  n,
\]
where $\omega$ stands for the clique number of $G.$ Note that Wilf's result
implies the concise Tur\'{a}n theorem:
\begin{equation}
e\left(  G\right)  \leq\left(  1-1/\omega\right)  n^{2}/2, \label{Ti}%
\end{equation}
where $e\left(  G\right)  $ denotes the number of edges of $G.$

Recently, at least three papers (\cite{CaFa09},\cite{HaLu09},\cite{YWS11})
appeared proving that if $G$ is a graph of order $n$, then%
\begin{equation}
q\left(  G\right)  \leq2\left(  1-1/\chi\right)  n, \label{Xi}%
\end{equation}
which, due to the known inequality $q\left(  G\right)  \geq2\mu\left(
G\right)  ,$ also improves (\ref{Ci}). This result suggests a natural
improvement of all above inequalities, and indeed, Hansen and Lucas
\cite{HaLu09} conjectured that the chromatic number $\chi$ in (\ref{Xi}) can
be replaced by the clique number $\omega.$

In this paper we prove the conjecture of Hansen and Lucas, as stated in the
following theorem.

\begin{theorem}
\label{mth}If $G$ is a graph of order $n$, with clique number $\omega\left(
G\right)  =\omega,$ then%
\begin{equation}
q\left(  G\right)  \leq2\left(  1-1/\omega\right)  n. \label{min}%
\end{equation}

\end{theorem}

The rest of the papers is organized as follows. We prove two supporting lemmas
of more general scope in the next section. In Section 3, a weaker form of the
inequality (\ref{min}) is proved. Finally, in Section 4 we state some results
about graph properties, which help us to complete the proof of Theorem
\ref{mth}.\medskip

Our notation follows essentially \cite{Bol98}. As usual, $G\left(  n\right)  $
stands for the set of graphs of order $n,$ and $K_{r}$ stands for the complete
graph of order $r.$ We write $\Gamma\left(  x\right)  $ for the set of
neighbors of a vertex $x.$

\section{Two lemmas}

Recall that a\emph{ book of size }$t$\emph{ }is a set of $t$ triangles sharing
a common edge.\medskip

Let us note first that Theorem \ref{mth} is rather easy for $\omega=2.$ In
fact, this case follows from a simple, but useful observation, given below.

\begin{lemma}
\label{le0} Every graph $G$ of order $n$ contains a book of size at least
$q\left(  G\right)  -n.$
\end{lemma}

\begin{proof}
Recall that in \cite{AnMo85}, Anderson and Morley gave a bound on the spectral
radius of the Laplacian matrix of a graph, whose proof implies also that
\[
q\left(  G\right)  \leq\max_{uv\in E\left(  G\right)  }d\left(  u\right)
+d\left(  v\right)  .
\]
Now let us select an edge $uv\in E\left(  G\right)  $ such that%
\[
d\left(  u\right)  +d\left(  v\right)  =\max_{uv\in E\left(  G\right)
}d\left(  u\right)  +d\left(  v\right)  .
\]
Obviously, the number of triangles containing the edge $uv$ is precisely equal
to $\left\vert \Gamma\left(  u\right)  \cap\Gamma\left(  v\right)  \right\vert
.$ By the inclusion-exclusion principle, we find that
\[
\left\vert \Gamma\left(  u\right)  \cap\Gamma\left(  v\right)  \right\vert
=\left\vert \Gamma\left(  u\right)  \right\vert +\left\vert \Gamma\left(
v\right)  \right\vert -\left\vert \Gamma\left(  u\right)  \cup\Gamma\left(
v\right)  \right\vert \geq d\left(  u\right)  +d\left(  v\right)  -n,
\]
and so, $G$ contains a book of size at least
\[
d\left(  u\right)  +d\left(  v\right)  -n\geq q\left(  G\right)  -n.
\]

\end{proof}

\begin{corollary}
If G is a triangle-free graph of order $n$, then $q\left(  G\right)  \leq n.$
\end{corollary}

In the following lemma we give a bound on the minimal entry of an eigenvector
to $q\left(  G\right)  ,$ which we shall need in the proof of Theorem
\ref{mth}, but which proved to be useful in other questions.

\begin{lemma}
\label{le1} Let $G$ be a graph of order $n$, with $q\left(  G\right)  =q,$ and
minimum degree $\delta\left(  G\right)  =\delta.$ If $\left(  x_{1}%
,\ldots,x_{n}\right)  $ is a unit eigenvector to $q,$ then the value $x_{\min
}=\min\left\{  x_{1},\ldots,x_{n}\right\}  $ satisfies the inequality%
\[
x_{\min}^{2}\left(  q^{2}-2q\delta+n\delta\right)  \leq\delta.
\]

\end{lemma}

\begin{proof}
If $x_{\min}=0,$ the assertion holds trivially, so suppose that $x_{\min}>0,$
which implies also that $\delta>0$. Now select $u\in V\left(  G\right)  $ so
that $d\left(  u\right)  =\delta.$ We have%
\[
qx_{u}=\delta x_{u}+%
{\textstyle\sum\limits_{i\in\Gamma\left(  u\right)  }}
x_{i}%
\]
and therefore,%
\begin{align*}
\left(  q-\delta\right)  ^{2}x_{\min}^{2}  &  \leq\left(  q-\delta\right)
^{2}x_{u}^{2}=\left(
{\textstyle\sum\limits_{i\in\Gamma\left(  u\right)  }}
x_{i}\right)  ^{2}\leq\delta%
{\textstyle\sum\limits_{i\in\Gamma\left(  u\right)  }}
x_{i}^{2}\leq\delta\left(  1-%
{\textstyle\sum\limits_{i\in V\left(  G\right)  \backslash\Gamma\left(
u\right)  }}
x_{i}^{2}\right) \\
&  \leq\delta\left(  1-\left(  n-\delta\right)  x_{\min}^{2}\right)
=\delta-\left(  \delta n-\delta^{2}\right)  x_{\min}^{2},
\end{align*}
implying that $x_{\min}^{2}\left(  q^{2}-2q\delta+n\delta\right)  \leq\delta,$
as claimed.
\end{proof}

\section{A weaker form of Theorem \ref{mth}}

To prove Theorem \ref{mth} we first prove a weaker form of it, which we shall
use later to deduce Theorem \ref{mth} itself. This approach can be used in
other extremal problems.

\begin{theorem}
\label{mthx}If $G$ is a graph of order $n$, with clique number $\omega\left(
G\right)  =\omega,$ then%
\[
q\left(  G\right)  \leq\frac{2\omega-2}{\omega}n+8.
\]

\end{theorem}

\begin{proof}
Fix an integer $r\geq2$ and set for short%
\[
q_{n}=\max_{G\in G\left(  n\right)  \text{ and }G\text{ is }K_{r+1}%
\text{-free}}q\left(  G\right)
\]
Clearly, to prove Theorem \ref{mthx} it is enough to show that
\begin{equation}
q_{n}\leq\frac{2r-2}{r}n+8. \label{in0}%
\end{equation}

Note that if $G$ is a $K_{3}$-free graph, then it has no book of positive size
and so Lemma \ref{le0} implies that $q\left(  G\right)  \leq n,$ proving the
theorem in this case. Thus, in the rest of the proof we shall assume that
$r\geq3.$ We shall proceed by induction of $n.$ If $n\leq r,$ it is known
that
\[
q\left(  G\right)  \leq2r-2\leq\frac{2r-2}{r}n,
\]
so the assertion holds whenever $n\leq r.$ Assume now that $n>r$ and that the
assertion holds for all $n^{\prime}<n.$ Assume for a contradiction that
inequality (\ref{in0}) fails, that is to say,
\begin{equation}
q_{n}>\frac{2r-2}{r}n+8, \label{in00}%
\end{equation}
and let $G\in G\left(  n\right)  $ be a $K_{r+1}$-free such that $q\left(
G\right)  =q_{n}.$ Let $\mathbf{x}=\left(  x_{1},\ldots,x_{n}\right)  $ be
unit eigenvector to $q_{n}$ and set $x=\min\left\{  x_{1},\ldots
,x_{n}\right\}  .$

For the sake of the reader the rest of our proof is split into several formal
claims.$\medskip$

\textbf{Claim 1. }%
\[
\left(  1-2x^{2}\right)  q_{n}\leq\left(  1-x^{2}\right)  q_{n-1}-nx^{2}+1
\]
\medskip

\emph{Proof. }\ Recall first that
\begin{equation}
q\left(  G\right)  =\left\langle Q\mathbf{x},\mathbf{x}\right\rangle
=\sum_{ij\in E\left(  G\right)  }\left(  x_{i}+x_{j}\right)  ^{2}. \label{Ral}%
\end{equation}
Let $u$ be a vertex for which $x_{u}=x,$ and write $G-u$ for the graph
obtained by removing the vertex $u.$ We have,
\begin{align}
q_{n}  &  =\sum_{ij\in E\left(  G\right)  }\left(  x_{i}+x_{j}\right)
^{2}=\sum_{ij\in E\left(  G-u\right)  }\left(  x_{i}+x_{j}\right)  ^{2}%
+\sum_{j\in\Gamma\left(  u\right)  }\left(  x+x_{j}\right)  ^{2}\nonumber\\
&  =\sum_{ij\in E\left(  G-u\right)  }\left(  x_{i}+x_{j}\right)
^{2}+d\left(  u\right)  x^{2}+2x\sum_{j\in\Gamma\left(  u\right)  }x_{j}%
+\sum_{j\in\Gamma\left(  u\right)  }x_{j}^{2} \label{in1}%
\end{align}
Hence, the Rayleigh principle implies that
\[
\left(  1-x^{2}\right)  q\left(  G-u\right)  \geq\sum_{ij\in E\left(
G-u\right)  }\left(  x_{i}+x_{j}\right)  ^{2},
\]
and since the graph $G-u$ is $K_{r+1}$-free, we see that
\begin{equation}
\sum_{ij\in E\left(  G-u\right)  }\left(  x_{i}+x_{j}\right)  ^{2}\leq\left(
1-x^{2}\right)  q\left(  G-u\right)  \leq\left(  1-x^{2}\right)  q_{n-1}.
\label{in2}%
\end{equation}
Now, using the equation
\[
\left(  q_{n}-d\left(  u\right)  \right)  x=\sum_{j\in\Gamma\left(  u\right)
}x_{j}%
\]
we find that%
\begin{align*}
d\left(  u\right)  x^{2}+2x\sum_{j\in\Gamma\left(  u\right)  }x_{j}+\sum
_{j\in\Gamma\left(  u\right)  }x_{j}^{2}  &  =d\left(  u\right)
x^{2}+2\left(  q_{n}-d\left(  u\right)  \right)  x^{2}+\sum_{j\in\Gamma\left(
u\right)  }x_{j}^{2}\\
&  \leq d\left(  u\right)  x^{2}+2\left(  q_{n}-d\left(  u\right)  \right)
x^{2}+1-\left(  n-d\left(  u\right)  \right)  x^{2}\\
&  =2q_{n}x^{2}-nx^{2}+1.
\end{align*}
This inequality, together with (\ref{in1}) and (\ref{in2}) implies the
claim.$\hfill\square\bigskip$

\textbf{Claim 2. }The following inequality holds%
\[
\left(  n+\frac{8r}{r-2}+\frac{2r-2}{r-2}\right)  x^{2}>1.
\]
\medskip

\emph{Proof. }By the induction assumption we have
\[
q_{n-1}\leq\frac{2r-2}{r}\left(  n-1\right)  +8.
\]
From Claim 1 and inequality (\ref{in00}), we obtain
\[
\left(  \frac{2r-2}{r}n+8\right)  \left(  1-2x^{2}\right)  <\left(
\frac{2r-2}{r}\left(  n-1\right)  +8\right)  \left(  1-x^{2}\right)
-nx^{2}+1,
\]
and so%
\[
\frac{2r-2}{r}n+8<\frac{2r-2}{r}\left(  n-1\right)  +8-\left(  \frac{2r-2}%
{r}\left(  n-1\right)  +8\right)  x^{2}+2\left(  \frac{2r-2}{r}n+8\right)
x^{2}-nx^{2}+1.
\]
Further simplifications give
\begin{align*}
\frac{r-2}{r}  &  =\frac{2r-2}{r}-1<-\left(  \frac{2r-2}{r}\left(  n-1\right)
+8\right)  x^{2}+2\left(  \frac{2r-2}{r}n+8\right)  x^{2}-nx^{2}\\
&  =\left(  \frac{r-2}{r}n+8+\frac{2r-2}{r}\right)  x^{2},
\end{align*}
and finally%
\[
1<\left(  n+\frac{8r}{r-2}+\frac{2r-2}{r-2}\right)  x^{2},
\]
as required.$\hfill\square\bigskip$

\textbf{Claim 3. }The following inequality holds\textbf{ }%
\[
\left(  \frac{4\left(  r-1\right)  }{r}n+16+\frac{8r}{r-2}+\frac{2r-2}%
{r-2}\right)  \delta>\frac{4\left(  r-1\right)  ^{2}}{r^{2}}n^{2}%
+\frac{32\left(  r-1\right)  }{r}n
\]
\medskip

\emph{Proof. }Using Lemma \ref{le1}, and the fact that
\begin{equation}
q^{2}-2q\delta+n\delta=\left(  q-\delta\right)  ^{2}+\delta\left(
n-\delta\right)  >0, \label{in4}%
\end{equation}
we first see that
\[
x^{2}\leq\frac{\delta}{q^{2}-2q\delta+n\delta}.
\]
Therefore, by Claim 2,
\[
q^{2}-2q\delta+n\delta\leq\left(  n+\frac{8r}{r-2}+\frac{2r-2}{r-2}\right)
\delta.
\]
Obviously, relation (\ref{in4}) implies also that $q^{2}-2q\delta+n\delta$
increases in $q,$ and so
\begin{align*}
\left(  n+\frac{8r}{r-2}+\frac{2r-2}{r-2}\right)  \delta &  \geq
q^{2}-2q\delta+n\delta\\
&  \geq\frac{4\left(  r-1\right)  ^{2}}{r^{2}}n^{2}+\frac{32\left(
r-1\right)  }{r}n+64-\frac{4\left(  r-1\right)  }{r}n\delta-16\delta+n\delta\\
&  >\frac{4\left(  r-1\right)  ^{2}}{r^{2}}n^{2}+\frac{32\left(  r-1\right)
}{r}n-\frac{4\left(  r-1\right)  }{r}n\delta-16\delta+n\delta.
\end{align*}
Collecting all terms involving $\delta$ on the left-hand side, we find that
\[
\left(  \frac{4\left(  r-1\right)  }{r}n+16+\frac{8r}{r-2}+\frac{2r-2}%
{r-2}\right)  \delta>\frac{4\left(  r-1\right)  ^{2}}{r^{2}}n^{2}%
+\frac{32\left(  r-1\right)  }{r}n,
\]
as claimed$\hfill\square\bigskip$

\textbf{Claim 4.} The theorem holds for $r\geq5.\medskip$\textbf{ }

\emph{Proof. }The Tur\'{a}n theorem (\ref{Ti}) implies that
\[
\delta\left(  G\right)  \leq\frac{2e\left(  G\right)  }{n}\leq\frac{\omega
-1}{\omega}n,
\]
This bound and Claim 3 imply that,
\begin{align*}
\frac{4\left(  r-1\right)  ^{2}}{r^{2}}n^{2}+\frac{32\left(  r-1\right)  }%
{r}n  &  <\left(  \frac{4\left(  r-1\right)  }{r}n+16+\frac{8r}{r-2}%
+\frac{2r-2}{r-2}\right)  \delta\\
&  \leq\left(  \frac{4\left(  r-1\right)  }{r}n+16+\frac{8r}{r-2}+\frac
{2r-2}{r-2}\right)  \frac{r-1}{r}n\\
&  =\frac{4\left(  r-1\right)  ^{2}}{r^{2}}n^{2}+\frac{16\left(  r-1\right)
}{r}n+\frac{8\left(  r-1\right)  }{r-2}n+\frac{2\left(  r-1\right)  ^{2}%
}{\left(  r-2\right)  r}n.
\end{align*}
After some algebra, we obtain consecutively
\begin{align*}
\frac{16\left(  r-1\right)  }{r}n  &  <\frac{8\left(  r-1\right)  }%
{r-2}n+\frac{2\left(  r-1\right)  ^{2}}{\left(  r-2\right)  r}n,\\
16\left(  r-2\right)   &  <8r+2\left(  r-1\right)  ,\\
r  &  <5.
\end{align*}
This is a contradiction for $r\geq5$, which proves the theorem for
$r\geq5.\hfill\square\bigskip$

It remains to prove the assertion for $r=3$ and $4,$ where we shall use Lemma
\ref{le0}. This lemma together with the assumption (\ref{in00}) implies that
$G$ contains a book of size at least
\begin{equation}
q_{n}-n>\frac{2r-2}{r}n+8-n=\frac{r-2}{r}n+8. \label{bW}%
\end{equation}

Select a book whose size satisfies (\ref{bW}), let $uv$ be the common edge of
this book, and write $W$ for set of common neighbors of $u$ and $v.$ Since we
assumed that $G$ contains no $K_{r+1}$, the graph $G\left[  W\right]  $
induced by the set $W$ contains no $K_{r-1}.$

\textbf{Claim 5 }The theorem holds for $r=3.\medskip$

\emph{Proof. }For $r=3$ the set $W$ is independent and by (\ref{bW}) we see
that
\[
\left\vert W\right\vert >\frac{n}{3}+8.
\]
A vertex in $W$ can be joined to at most $n-\left\vert W\right\vert $ vertices
of $G$ and so
\[
\delta\leq n-\left\vert W\right\vert \leq\frac{2}{3}n-8.
\]
Using Claim 3, we obtain
\begin{align*}
\frac{16}{9}n^{2}+\frac{64}{3}n  &  <\left(  \frac{8}{3}n+16+24+4\right)
\left(  \frac{2}{3}n-8\right) \\
&  =8n+\frac{16}{9}n^{2}-352,
\end{align*}
and so,
\[
\frac{40}{3}n\leq-352,
\]
a contradiction proving the theorem for $r=3.\hfill\square\bigskip$

\textbf{Claim 6 }The theorem holds for $r=4.\medskip$

\emph{Proof. }For $r=4,$ the set $W$ induces a triangle-free graph by
(\ref{bW}) we have
\[
\left\vert W\right\vert >\frac{n}{2}+8.
\]
According to (\ref{Ti}), the minimum degree of the graph $G\left[  W\right]  $
induced by $W$ is at most $\left\vert W\right\vert /2.$ Let $u\in W$ be a
vertex with minimum degree in $G\left[  W\right]  .$ Since $u$ is joined to at
most $n-\left\vert W\right\vert $ vertices outside of $W,$ we see that
\[
\delta\leq\frac{\left\vert W\right\vert }{2}+n-\left\vert W\right\vert
=n-\frac{\left\vert W\right\vert }{2}<n-\frac{n}{4}-4=\frac{3}{4}n-4.
\]
Again by Claim 3, we obtain%
\[
\frac{9}{4}n^{2}+24n<\left(  \frac{12}{4}n+35\right)  \left(  \frac{3}%
{4}n-4\right)  =\frac{9}{4}n^{2}+\frac{57}{4}n+-140
\]
and so%
\[
\frac{39}{4}n<-140
\]
a contradiction proving the theorem for $r=4.\hfill\square\bigskip$

At this point all cases have been considered, the induction step is completed
and Theorem \ref{mthx} is proved.
\end{proof}

\section{Graph properties and proof of Theorem \ref{mth}}

Before continuing with the proof of Theorem \ref{mth}, we introduce relevant
results about graph properties.

Recall that a \emph{graph property} $\mathcal{P}$ is a class of graphs closed
under isomorphisms. A property $\mathcal{P}$ is called \emph{hereditary} if it
is closed under taking \emph{induced} subgraphs.\medskip

Given a graph property $\mathcal{P}$, we write $\mathcal{P}_{n}$ for the set
of graphs of order $n$ belonging to $\mathcal{P}$ and we set
\[
q\left(  \mathcal{P}_{n}\right)  =\max_{G\in\mathcal{P}_{n}}\text{ }q\left(
G\right)  .
\]

\begin{theorem}
\label{limT}If $\mathcal{P}$ is a hereditary graph property, then the limit
\[
\lim_{n\rightarrow\infty}\frac{q\left(  \mathcal{P}_{n}\right)  }{n}%
\]
exists.
\end{theorem}

Given a hereditary property $\mathcal{P},$ write $\nu(\mathcal{P})$ for the
limit established in Theorem \ref{limT}.\medskip

For any graph $G$ and integer $t\geq1,$ write $G^{\left(  t\right)  }$ for the
graph obtained by replacing each vertex $u$ of $G$ by a set of $t$ independent
vertices and each edge $uv$ of $G$ by a complete bipartite graph $K_{t,t}.$
This construction is known as a \textquotedblleft blow-up\textquotedblright%
\ of $G.$

A graph property $\mathcal{P}$ is called\emph{ multiplicative }if
$G\in\mathcal{P}$ implies that $G^{\left(  t\right)  }\in\mathcal{P}$ for all
$t\geq1.$\medskip

We shall deduce the proof of Theorem \ref{mth} from the following theorem,
proved in \cite{AbNi11}.

\begin{theorem}
\label{mT}If $\mathcal{P}$ is a hereditary and multiplicative graph property,
then
\[
q(G)\leq\nu\left(  \mathcal{P}\right)  \left\vert G\right\vert
\]
for every $G\in\mathcal{P}.$
\end{theorem}

Easy arguments show that the class $\mathcal{P}^{\ast}\left(  K_{r+1}\right)
$ of graphs containing no $K_{r+1}$ is hereditary, unbounded and
multiplicative. Therefore, Theorem \ref{mT} implies that%
\begin{equation}
q(G)\leq\nu\left(  \mathcal{P}^{\ast}\left(  K_{r+1}\right)  \right)
\left\vert G\right\vert \label{in5}%
\end{equation}
for every $G\in\mathcal{P}^{\ast}\left(  K_{r+1}\right)  .$ But Theorem
\ref{mthx} implies that%
\[
\nu\left(  \mathcal{P}^{\ast}\left(  K_{r+1}\right)  \right)  \leq2\left(
1-1/r\right)
\]
and this inequality together with (\ref{in5}) completes the proof of Theorem
\ref{mth}.\bigskip

\textbf{Concluding remarks}\medskip

It is not hard to prove that for $\omega=2,$ equality in (\ref{min}) holds if
and only if $G$ is a complete bipartite graph. Also, for $\omega\geq3$
equality holds if $G$ is a regular complete $\omega$-partite graph. It seems
quite plausible that for $\omega\geq3$ this is the only case for equality in
(\ref{min}), but our proof does not give immediate evidence of this fact.

The result of this paper was presented at the workshop \textquotedblleft
Linear Algebraic Techniques in Combinatorics/Graph Theory\textquotedblright%
\ in February, 2011, at BIRS, Canada. Recently, an elegant proof of a slightly
more precise result was obtained independently by B. He, Y.L. Jin and X.D.
Zhang \cite{HJZ11}. We are grateful to these authors for sharing their
manuscript.\bigskip

\textbf{Acknowledgement}\medskip

This work was started in 2010 when the second author was visiting the Federal
University of Rio de Janeiro, Brazil. He is grateful for the warm hospitality
of the Spectral Graph Theory Group in Rio de Janeiro.

\end{document}